\documentclass{amsart}
\usepackage{amsthm}
\usepackage{pinlabel}
\usepackage{color}
\usepackage{algorithm}
\usepackage{subcaption}
\usepackage{tikz}
\usetikzlibrary{arrows.meta, positioning}
\usepackage{graphicx}
\usepackage{multicol}
\usepackage{enumerate}
\usepackage{hyperref}\hypersetup{colorlinks,allcolors=black}
\usepackage[capitalize,noabbrev]{cleveref}
\crefname{alphatheorem}{Theorem}{Theorems}

\usepackage{listings}
\usepackage{xcolor}

\newtheorem{theorem}{\bf Theorem}[section]
\newtheorem{lemma}[theorem]{\bf Lemma}
\newtheorem{corollary}[theorem]{\bf Corollary}
\newtheorem{proposition}[theorem]{\bf Proposition}

\newcounter{alpharesults}

\newtheorem{alphaconjecture}[alpharesults]{Conjecture}

\theoremstyle{definition}
\newtheorem{definition}[theorem]{Definition}
\newtheorem{remark}[theorem]{Remark}

\begin{document}

\title{Mutual arc presentations and braided open books}

\author{Benjamin Bode}
\address{Departamento de Matemática Aplicada a la Ingeniería Industrial, ETSIDI, Universidad Politécnica de Madrid, Rda. de Valencia 3, 28012 Madrid, Spain}
\email{benjamin.bode@upm.es}

\author{Chun-Sheng Hsueh}
\address{Department of Mathematics, National Taiwan University, Taipei 10617, Taiwan}
\email{chunshenghsueh@ntu.edu.tw}





\begin{abstract}
    We show that every canonically fibered link in $S^3$ is the binding of a braided open book in $S^3$, addressing a question of Montesinos and Morton. We introduce mutual arc presentations as our main technical tool, which we consider to be of independent interest. We prove that any fibered link admitting such a presentation is the binding of a braided open book.
\end{abstract}

\maketitle

\section{Introduction}
A \emph{fibered link} $B$ in $S^3$ is an oriented link that is the binding of an open book in $S^3$; equivalently, its complement fibers over the circle, and the closure of each fiber is a Seifert surface with boundary orientation agreeing with the link orientation. 

This paper is motivated by the search for an alternative (more elementary and topological) proof of the Harer Conjecture~\cite{Harer}, which is now Theorem~\ref{thm:harer}. The existing proof relies on deep results from contact geometry, including Giroux's correspondence between contact structures and open books~\cite{girouxcorrespondencearbitrarydimensions, giroux, girouxcorrespondencedimension3} and the classification of overtwisted contact structures on $3$-manifolds by Eliashberg~\cite{Eliashberg}.

\begin{theorem}[\cite{giroux_goodman}]\label{thm:harer}
Every fibered link in $S^3$ can be obtained from the unknot and its fiber disk by a sequence of Hopf plumbings and deplumbings.
\end{theorem}

Moreover, we are interested in different potential strengthenings of Theorem~\ref{thm:harer}. 
Montesinos and Morton observed that Theorem~\ref{thm:harer} is implied by Conjecture~\ref{conj:mm}.

\begin{alphaconjecture}[\cite{montesions_morton}]\label{conj:mm}
For every fibered link $B$ there is a simple branched cover $\pi \colon S^{3} \to S^{3}$ branched over a link $L$ and two unknots $U_1$ and $U_2$, which are braid axes of $L$, such that $\pi^{-1}(U_1) = B$ and $\pi^{-1}(U_2)$ is the unknot.
\end{alphaconjecture}

Montesinos and Morton proved Conjecture~\ref{conj:mm} for all fibered links that can be obtained from the unknot and its fiber disk by a sequence of Hopf plumbings (without any deplumbings)~\cite{montesions_morton}.

The property in Conjecture~\ref{conj:mm} is related to one of four possible notions of what it means for an open book in $S^3$ to be braided, which have been shown to be equivalent to each other~\cite{bode:braided}. The most natural of these four definitions goes as follows. The \emph{unbook} is the standard open book in $S^3$ with unknotted binding, and an open book in $S^3$ is said to be \emph{braided} if there exists an unknot $U$ that is braided relative to it, meaning that $U$ is positively transverse to every page of the open book, and, simultaneously, the binding of the given open book is positively braided relative to the unbook with binding $U$. 

Equivalently, the bindings of braided open books correspond precisely to the closures of $P$-fibered braids. A braid $\beta$ is \emph{$P$-fibered} if it can be represented by a loop of polynomials whose argument map $\arg g \colon (\mathbb{C}\times S^{1}) \setminus \beta \longrightarrow S^{1}$ induces a fibration.

\begin{alphaconjecture}[\cite{bode:braided}]\label{conj:mutual_braid}
Every fibered link in $S^3$ is the binding of a braided open book, or equivalently, the closure of a $P$-fibered braid.
\end{alphaconjecture}

The first-named author~\cite{bode:braided} showed that Conjecture~\ref{conj:mm} is implied by Conjecture~\ref{conj:mutual_braid} and proved Conjecture~\ref{conj:mutual_braid} for fibered links with braid index of at most $3$ and closures of homogeneous braids (the closure of $P$-fibered braids are always fibered as are the closures of homogeneous braids). In fact, the links that satisfy Conjecture~\ref{conj:mutual_braid} satisfy a property that a priori is stronger than what is stated in Theorem~\ref{thm:harer}: All of their plumbing arcs can be taken to have a certain symmetry with respect to a fixed simple branched cover $\pi\colon S^3\to S^3$ as in Conjecture~\ref{conj:mm}. Therefore, a proof of Conjecture~\ref{conj:mutual_braid} would not only give a new proof of Theorem~\ref{thm:harer}, but would strengthen it by showing that the plumbing arcs can be arranged to satisfy this additional symmetry.

This paper provides a novel technique for proving that certain fibered links are bindings of braided open books and thus satisfy both Conjecture~\ref{conj:mutual_braid} and Conjecture~\ref{conj:mm}.
To this end, we introduce \emph{mutual arc presentations}, see Definition~\ref{def:mutual arc}, generalizing arc presentations defined by Cromwell~\cite{Cromwell}.
\begin{proposition}\label{prop:mutual book presentation to mutually braided}
    If a fibered link in $S^3$ admits a mutual arc presentation, then it satisfies Conjecture~\ref{conj:mutual_braid}. 
\end{proposition}

In particular, the following conjecture implies Conjecture~\ref{conj:mutual_braid}.

\begin{alphaconjecture}\label{conj:mutual_arc}
Every fibered link in $S^3$ admits a mutual arc presentation.
\end{alphaconjecture}

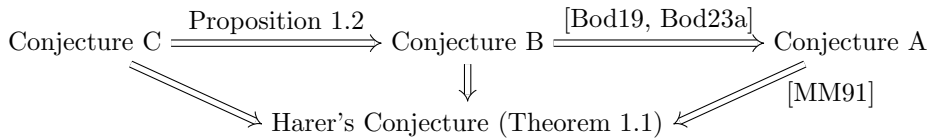
\begin{figure}[h!]
    \newcommand{\Description}[1]{}
    \centering
        \begin{tikzpicture}[
            box/.style={draw, rectangle, rounded corners, minimum width=2.5cm, minimum height=0.7cm, align=center, font=\small},
            label/.style={font=\scriptsize, align=center}
        ]
        
        \node[] (C1) {Conjecture~\ref{conj:mutual_arc}};
        \node[right=2.8cm of C1] (C2) {Conjecture~\ref{conj:mutual_braid}};
        \node[right=2.8cm of C2] (C3) {Conjecture~\ref{conj:mm}};
        \node[below=0.5cm of C2] (T) {Harer's Conjecture (Theorem~\ref{thm:harer})};

        \draw[-Implies,double distance=2pt,  thin] (C3) -- (T.east) node[midway, right=0.5cm]{\cite{montesions_morton}};
        \draw[-Implies, double distance=2pt, thin] (C1) -- (C2) node[midway, above=0cm] {Proposition~\ref{prop:mutual book presentation to mutually braided}};
        \draw[-Implies, double distance=2pt, thin] (C2) -- (C3) node[midway, above=0cm] {\cite{Bode2019, bode:braided}};
        
        \end{tikzpicture}
    \caption{Implications among conjectures and Theorem~\ref{thm:harer}.}
    \Description{A flowchart showing implications among three conjectures and Harer's Conjecture. The first conjecture implies the second by Proposition 1.2; the second implies the third by Bode; and the third implies Harer's Conjecture (Theorem 1.1) by Montesinos and Morton.}
    \label{fig:flowchart}
\end{figure}

For short, we say a fibered link in $S^3$ has \emph{Property $A$ (resp.\ B, or C)} if it is not a counterexample to Conjecture~\ref{conj:mm} (resp.~\ref{conj:mutual_braid}, or~\ref{conj:mutual_arc}). 

A large subclass of fibered links is \emph{canonically} fibered, meaning that their fiber surfaces arise from applying Seifert’s algorithm to a link diagram. Using our new technique we prove that all links from this infinite family have Property~C.

\begin{theorem}\label{thm:gc}
All canonically fibered links in $S^3$ have Property~C.
\end{theorem}

It can be deduced from~\cite{GHY, montesions_morton} that alternating and almost alternating fibered links have Property~A. While it remains unknown if all almost alternating fibered have Properties~B or~C, we obtain the following using~\cite{alternating1, alternating2, Stoimenow}.
\begin{corollary}\label{cor:alternating}
All alternating fibered links in $S^3$ have Properties~A,~B, and~C. Moreover, there are infinitely many almost alternating fibered links in $S^3$ that have Properties~A,~B, and~C.
\end{corollary}

Cromwell proved that applying Seifert's algorithm to a connected, homogeneous diagram produces a minimal genus Seifert surface~\cite[Corollary~4.1]{cromwell2}. Therefore, all homogeneous fibered links have Properties A, B and C. This family contains the alternating fibered links from Corollary~\ref{cor:alternating}.

Using a family of fibered links constructed by Stoimenow~\cite{Stoimenow_infinite_family}, we obtain the first example of a fibered link that satisfies Property~A but is not a Hopf plumbing~\cite{Misev_2016}, i.e.\ it cannot be constructed from the unknot via a sequence of Hopf plumbings and deplumbings without performing Hopf deplumbing.

\begin{corollary}\label{cor:deplumbing}
There exists an infinite family of fibered links having Properties~A,~B, and~C that cannot be obtained from the unknot through Hopf plumbings.
\end{corollary}

\subsection*{Outline}
In Section~\ref{sec:def}, we recall the classical definitions of braids and arc presentations. While braided open books generalize braids, we introduce \emph{mutual arc presentations} as the natural generalization of arc presentations. Proposition~\ref{prop:mutual book presentation to mutually braided}, which establishes a connection between these two concepts, is presented. Section~\ref{sec:canonical} provides an algorithm for constructing mutual arc presentations of canonically fibered links, proving Theorem~\ref{thm:gc}.

\subsection*{Acknowledgment}
The first-named author is grateful to the Berlin Mathematics Research Center MATH+ and Marc Kegel for making his visits to Berlin possible.

The second-named author would like to thank his advisor, Marc Kegel, for his support and the Universidad Polit\'ecnica de Madrid for its hospitality during his visit, which was funded by the Humboldt-Universit\"at zu Berlin. He gratefully acknowledges support from the Claussen-Simon-Stiftung through the Dissertation Plus Scholarship and is a member of the Berlin Mathematics Research Center MATH+ (EXC-2046/1, project ID: 390685689), funded by the Deutsche Forschungsgemeinschaft (DFG) under Germany's Excellence Strategy.

We thank the anonymous referee for their helpful feedback, which enhanced the clarity of this paper, and Pedro Gonz\'{a}lez Manch\'{o}n for pointing out Cromwell's result on homogeneous links.

\section{From braided open books to mutual arc presentations}\label{sec:def}
The notion of an arc presentation is conceptually related to that of a braid. While a braided oriented link is positively transverse to the pages of the unbook, an arc presentation is tangent to the pages of the unbook. Like in the case of braids~\cite{alexander}, it is known that every link admits an arc presentation~\cite{Cromwell}. The generalization of a braid to braidings relative to open books other than the unbook allows the extension to more general $3$-manifolds~\cite{ItoKawamuro,Pavelescu1, Pavelescu2} and the concept of mutual braiding~\cite{generalized_exchangeable, generalized_exchangeable2}. We generalize the definition of an arc presentation and introduce a notion of a mutual arc presentation, where two links are simultaneously (generalized) arc presentations relative to each other.

\subsection{Braided open books}
Let $M^3$ be a closed, oriented $3$-manifold. Given an embedded, oriented surface $\mathcal{S}$ in $M$ and an oriented link $L$ in $M$, we say that $L$ is \emph{positively transverse} to $\mathcal{S}$ if $L$ and $\mathcal{S}$ are transverse and the sign of every intersection point is positive.

An oriented link $L$ in $S^3$ is \emph{braided relative to} an unknot $U\subset S^3\backslash L$ if $L$ is positively transverse to all pages of the unbook with binding $U$. In this case, we say $U$ is a \emph{braid axis} for $L$. More generally, let $\operatorname{OB}$ be an open book decomposition on $M$, we say $L$ is \emph{braided relative to} (the binding of) $\operatorname{OB}$ if $L$ is disjoint from the binding of $OB$ and positively transverse to the pages of $\operatorname{OB}$.

\begin{definition}\label{def:braided}
    Let $B_1,B_2\subset M$ be fibered links in $M$ and let $\operatorname{OB}_1$ and $\operatorname{OB}_2$ be open book decompositions on $M$ with bindings $B_1$ and $B_2$, respectively. 
    \begin{enumerate}[(i)]
        \item We say that $B_1$ and $B_2$ are \emph{mutually braided} if $B_1$ is braided relative to $\operatorname{OB}_2$ and $B_2$ is braided relative to $\operatorname{OB}_1$. We also say $\operatorname{OB}_1$ and $\operatorname{OB}_2$ are \emph{mutually braided} in this case.
        \item We say an open book $\operatorname{OB}$ in $S^3$ is \emph{braided} if $\operatorname{OB}$ and the unbook are mutually braided.
    \end{enumerate}
\end{definition}

If a fibered link $B$ in $S^3$ and an unknot are mutually braided, they are also said to be \emph{generalized exchangeable}~\cite{generalized_exchangeable, generalized_exchangeable2}.


There are other definitions of a braided open book in $S^3$, which have been shown to be equivalent to this one~\cite{bode:braided}. It is not known if every open book in $S^3$ can be braided.

\subsection{Mutual arc presentations}
An \emph{arc presentation} of a link $L\subset S^3$ is a representative of its isotopy class where each connected component intersects the unknot $U$ in a finite number of times and such that each segment of $L$, i.e.\ each arc between two intersection points with $U$, lies on a page of the unbook with binding $U$.

\begin{definition}\label{def:mutual arc}
Let $B_1,B_2\subset M$ be fibered links in $M$ and let $\operatorname{OB}_1$ and $\operatorname{OB}_2$ be open book decompositions on $M$ with bindings $B_1$ and $B_2$, respectively. 
    \begin{enumerate}[(i)]
        \item A \emph{generalized arc presentation} of $B_1$ relative to $\operatorname{OB}_2$ is a representative of the isotopy class of $B_1$ that intersects $B_2$ in finitely many points and such that each segment, i.e.\ each arc between two intersection points with $B_2$, lies on the page of $\operatorname{OB}_2$.
        \item We say that $B_1$ and $B_2$ have a \emph{mutual arc presentation} if $B_1$ is a generalized arc presentation relative to $\operatorname{OB}_2$ and $B_2$ is a generalized arc presentation relative to $\operatorname{OB}_1$. We also say $\operatorname{OB}_1$ and $\operatorname{OB}_2$ have a \emph{mutual arc presentation} in this case.
        \item We say the binding $B$ of an open book $\operatorname{OB}$ in $S^3$ admits a \emph{mutual arc presentation} if $\operatorname{OB}$ and the unbook have a mutual arc presentation.
    \end{enumerate}
\end{definition}

If $M=S^3$ and $\operatorname{OB}$ is an unbook, then the notion of a generalized arc presentation relative to $\operatorname{OB}$ is exactly the usual concept of an arc presentation.

Note that in (i) of Definition~\ref{def:mutual arc}, we did not use the fiberedness of $B_1$. Moreover, we used the notion of open books to define generalized braids and generalized arc presentations. However, the latter admit a more general formulation that does not rely on open books, but is instead expressed in terms of the closely related notion of half open books. 

\begin{definition}[cf. Definition~3.1 in~\cite{hsueh}]\label{def:halfopen}
Given $\mathcal{S} \subset M$ an embedded, connected, compact, oriented surface with boundary, let $\tau\colon\mathcal{S}\times[-1,1]\to M$ be a map such that
\begin{itemize}
\item $\tau(p,0)=p$ for all $p\in \mathcal{S}$,
\item $\tau(p,t)=p$ for all $p\in \partial \mathcal{S}$, and
\item the restriction of $\tau$ to $\text{int}(\mathcal{S})\times[-1,1]$ is a smooth embedding.
\end{itemize}
The image $H\mathcal{S}=\tau(\mathcal{S}\times[-1,1])$ together with its structure as a fibration $H\mathcal{S}\setminus \partial\mathcal{S} \to [-1, 1]$ given by the parameter $t$ is called a \emph{half open book}.\\
We adopt the terminology of open books, i.e., the closure of fibers of the fibration map are called \emph{pages} and $\partial\mathcal{S}$ is called the \emph{binding}.
\end{definition}

Definition~\ref{def:halfopen} generalizes the definition of a lens thickening of a Seifert surface in~\cite{baader}. In particular, every Seifert surface is the page of a half open book. 

Let $\mathcal{S}_0$ be the page of an open book on $M$ (in particular, the closure of a fiber of $\varphi\colon M\setminus \partial\mathcal{S}_0 \to S^1$). Since every closed, simply connected subset $V$ of $S^1$ is homeomorphic to the interval $[-1,1]$ we may consider $\varphi^{-1}(V)$ as a half open book with $\mathcal{S}_0$ as the closure of the fiber over $0\in[-1,1]$. The pages of the half open book are exactly those pages of the open book that are over $V$.



\begin{definition}[Generalization of Definition~\ref{def:mutual arc}]\label{def: generalized mutual arc}
    Let $L_1,L_2\subset M$ be oriented links and let $\mathcal{S}_1,\mathcal{S}_2\subset M$ be embedded, connected, compact, oriented surfaces with boundaries $L_1$ and $L_2$, respectively.
    \begin{enumerate}[(i)]
        \item A \emph{generalized arc presentation} of $L_1$ relative to $\mathcal{S}_2$ is a representative of the isotopy class of $L_1$ that intersects $L_2$ in finitely many points and such that each segment, i.e.\ each arc between two intersection points with $L_2$, lies on some page of a fixed half open book with page $\mathcal{S}_2$.
        \item We say that $L_1$ and $L_2$ have a \emph{mutual arc presentation} relative to $\mathcal{S}_2$ and $\mathcal{S}_1$ if $L_1$ is a generalized arc presentation relative to $\mathcal{S}_2$ and $L_2$ is a generalized arc presentation relative to $\mathcal{S}_1$.
    \end{enumerate}
\end{definition}

In the case where all links involved are bindings of open books and the surfaces $\mathcal{S}_1$ and $\mathcal{S}_2$ are pages of the corresponding open books, the pages of the half open books of $\mathcal{S}_1$ and $\mathcal{S}_2$ can also be taken to be pages of the respective open books. Thus, in this case, Definition~\ref{def:mutual arc} (ii) and Definition~\ref{def: generalized mutual arc} (ii) are equivalent, that is, two open books with bindings $B_1$ and $B_2$ and pages $\mathcal{S}_1$ and $\mathcal{S}_2$, respectively, have a mutual arc presentation if and only if $B_1$ and $B_2$ have a mutual arc presentation relative to $\mathcal{S}_2$ and $\mathcal{S}_1$.


\subsection{Back to braided open books}
The following result shows that Conjecture~\ref{conj:mutual_braid} on braided open books is implied by our Conjecture~\ref{conj:mutual_arc} on mutual arc presentations, proving Proposition~\ref{prop:mutual book presentation to mutually braided} in a more general setting.
\begin{proposition}\label{prop: detailed mutual book presentation to mutually braided}
    Given a generalized arc presentation of an oriented link $L\subset M^3$ relative to an open book $\operatorname{OB}$ on $M$, the link $L$ can be isotoped to become braided relative to $\operatorname{OB}$. Moreover, given a mutual arc presentation of fibered links $B_1$ and $B_2\subset M$, then $B_1$ and $B_2$ can be isotoped into mutually braided links.
\end{proposition}
\begin{remark}
The isotopies of $B_1$ and $B_2$ are two separate isotopies, not one isotopy of $B_1\cup B_2$. Otherwise, the result would still have intersection points between the two links. Instead, we may first deform $B_1$ into a link that is braided relative to the open book $\operatorname{OB}_2$ with binding $B_2$ and then deform $B_2$ into a link that is braided relative to the open book $\operatorname{OB}_1$ with binding $B_1$.
\end{remark}
\begin{proof}
Near each intersection point of $L$ with the binding $B$ of the given open book $\operatorname{OB}$, the link $L$ intersects $B$ transversely, as shown schematically in Figure~\ref{fig:arc presentation}. Choose a small disk centered at the intersection point, containing a local segment of $L$ and transverse to $B$. The orientation of $B$ induces an orientation on the boundary of the disk via the right-hand rule. Observe that the behavior of the pages of an open book near its binding implies that the boundary of the disk is positively transverse to the pages of $\operatorname{OB}$.

Consider an arc of $L$ connecting two intersection points. The corresponding pair of intersections of $L$ with $B$ have opposite signs. At each intersection point, isotope the local segment of $L$ to the boundary of the corresponding disk so that their orientations agree, while keeping the remainder of $L$ fixed. Since $L$ is in a generalized arc presentation relative to $\operatorname{OB}$, every arc outside the chosen neighborhoods remains contained in a page of $\operatorname{OB}$ throughout the isotopy, while each modified local segment becomes positively transverse to the pages of $\operatorname{OB}$, as illustrated in Figure~\ref{fig:arc on the page}. Finally, by a small perturbation as in Figure~\ref{fig:small perturbation} of each remaining arc lying on a page, $L$ becomes positively transverse to the pages of $\operatorname{OB}$. Hence $L$ is braided relative to $\operatorname{OB}$.
    
    \begin{figure}[ht]
    \newcommand{\Description}[1]{}
        \begin{subfigure}{0.28\textwidth}
                \renewcommand\captionlabelfont{}
                \centering
                \def\svgwidth{0.75\textwidth}
\begingroup%
  \makeatletter%
  \providecommand\color[2][]{%
    \errmessage{(Inkscape) Color is used for the text in Inkscape, but the package 'color.sty' is not loaded}%
    \renewcommand\color[2][]{}%
  }%
  \providecommand\transparent[1]{%
    \errmessage{(Inkscape) Transparency is used (non-zero) for the text in Inkscape, but the package 'transparent.sty' is not loaded}%
    \renewcommand\transparent[1]{}%
  }%
  \providecommand\rotatebox[2]{#2}%
  \newcommand*\fsize{\dimexpr\f@size pt\relax}%
  \newcommand*\lineheight[1]{\fontsize{\fsize}{#1\fsize}\selectfont}%
  \ifx\svgwidth\undefined%
    \setlength{\unitlength}{185.51383912bp}%
    \ifx\svgscale\undefined%
      \relax%
    \else%
      \setlength{\unitlength}{\unitlength * \real{\svgscale}}%
    \fi%
  \else%
    \setlength{\unitlength}{\svgwidth}%
  \fi%
  \global\let\svgwidth\undefined%
  \global\let\svgscale\undefined%
  \makeatother%
  \begin{picture}(1,1.1441843)%
    \lineheight{1}%
    \setlength\tabcolsep{0pt}%
    \put(0,0){\includegraphics[width=\unitlength,page=1]{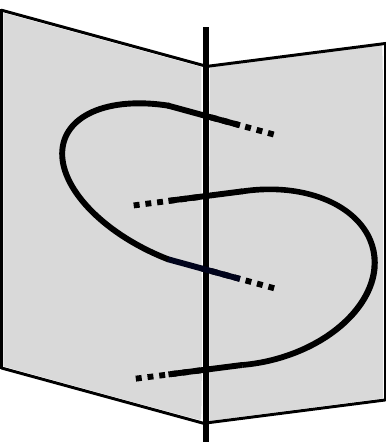}}%
    \put(0.46498415,1.08611069){\makebox(0,0)[lt]{\lineheight{1.25}\smash{\begin{tabular}[t]{l}$B$\end{tabular}}}}%
    \put(0.0416705,0.0336245){\makebox(0,0)[lt]{\lineheight{1.25}\smash{\begin{tabular}[t]{l}$\operatorname{OB}$\end{tabular}}}}%
    \put(0.85033106,0.65623345){\makebox(0,0)[lt]{\lineheight{1.25}\smash{\begin{tabular}[t]{l}$L$\end{tabular}}}}%
  \end{picture}%
\endgroup%

                \caption{}
                \Description{The binding and two pages of an open book are shown, each containing a properly embedded arc.}
                \label{fig:arc presentation}
        \end{subfigure}
            \hfill
        \begin{subfigure}{0.28\textwidth}
                \renewcommand\captionlabelfont{}
                \centering
                \def\svgwidth{0.75\textwidth}
\begingroup%
  \makeatletter%
  \providecommand\color[2][]{%
    \errmessage{(Inkscape) Color is used for the text in Inkscape, but the package 'color.sty' is not loaded}%
    \renewcommand\color[2][]{}%
  }%
  \providecommand\transparent[1]{%
    \errmessage{(Inkscape) Transparency is used (non-zero) for the text in Inkscape, but the package 'transparent.sty' is not loaded}%
    \renewcommand\transparent[1]{}%
  }%
  \providecommand\rotatebox[2]{#2}%
  \newcommand*\fsize{\dimexpr\f@size pt\relax}%
  \newcommand*\lineheight[1]{\fontsize{\fsize}{#1\fsize}\selectfont}%
  \ifx\svgwidth\undefined%
    \setlength{\unitlength}{185.51383912bp}%
    \ifx\svgscale\undefined%
      \relax%
    \else%
      \setlength{\unitlength}{\unitlength * \real{\svgscale}}%
    \fi%
  \else%
    \setlength{\unitlength}{\svgwidth}%
  \fi%
  \global\let\svgwidth\undefined%
  \global\let\svgscale\undefined%
  \makeatother%
  \begin{picture}(1,1.13307744)%
    \lineheight{1}%
    \setlength\tabcolsep{0pt}%
    \put(0,0){\includegraphics[width=\unitlength,page=1]{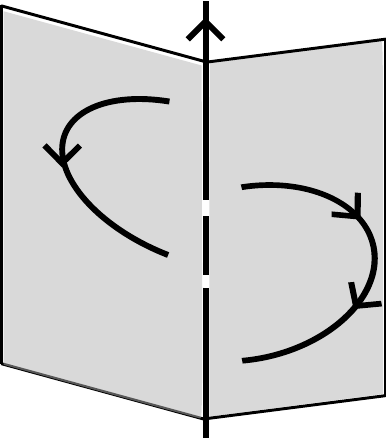}}%
    \put(0.5674144,1.07500382){\makebox(0,0)[lt]{\lineheight{1.25}\smash{\begin{tabular}[t]{l}$B$\end{tabular}}}}%
    \put(0.82954761,0.69734627){\makebox(0,0)[lt]{\lineheight{1.25}\smash{\begin{tabular}[t]{l}$L$\end{tabular}}}}%
    \put(0.0416705,0.03362449){\makebox(0,0)[lt]{\lineheight{1.25}\smash{\begin{tabular}[t]{l}$\operatorname{OB}$\end{tabular}}}}%
    \put(0,0){\includegraphics[width=\unitlength,page=2]{arc_presentation_to_braid.pdf}}%
  \end{picture}%
\endgroup%

                \caption{}
                \Description{The properly embedded arcs are pushed off from the pages near the binding.}
                \label{fig:arc on the page}
        \end{subfigure}
            \hfill
        \begin{subfigure}{0.35\textwidth}
                \renewcommand\captionlabelfont{}
                \centering
                \def\svgwidth{0.65\textwidth}
\begingroup%
  \makeatletter%
  \providecommand\color[2][]{%
    \errmessage{(Inkscape) Color is used for the text in Inkscape, but the package 'color.sty' is not loaded}%
    \renewcommand\color[2][]{}%
  }%
  \providecommand\transparent[1]{%
    \errmessage{(Inkscape) Transparency is used (non-zero) for the text in Inkscape, but the package 'transparent.sty' is not loaded}%
    \renewcommand\transparent[1]{}%
  }%
  \providecommand\rotatebox[2]{#2}%
  \newcommand*\fsize{\dimexpr\f@size pt\relax}%
  \newcommand*\lineheight[1]{\fontsize{\fsize}{#1\fsize}\selectfont}%
  \ifx\svgwidth\undefined%
    \setlength{\unitlength}{152.96312227bp}%
    \ifx\svgscale\undefined%
      \relax%
    \else%
      \setlength{\unitlength}{\unitlength * \real{\svgscale}}%
    \fi%
  \else%
    \setlength{\unitlength}{\svgwidth}%
  \fi%
  \global\let\svgwidth\undefined%
  \global\let\svgscale\undefined%
  \makeatother%
  \begin{picture}(1,1.20493231)%
    \lineheight{1}%
    \setlength\tabcolsep{0pt}%
    \put(0,0){\includegraphics[width=\unitlength,page=1]{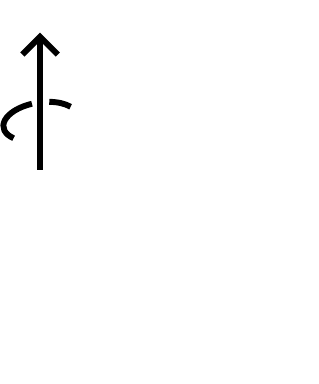}}%
    \put(0.04452584,1.13450056){\makebox(0,0)[lt]{\lineheight{1.25}\smash{\begin{tabular}[t]{l}$B$\end{tabular}}}}%
    \put(0.28591657,0.26691276){\makebox(0,0)[lt]{\lineheight{1.25}\smash{\begin{tabular}[t]{l}$L$\end{tabular}}}}%
    \put(0,0){\includegraphics[width=\unitlength,page=2]{braid.pdf}}%
    \put(0.78065014,1.13450056){\makebox(0,0)[lt]{\lineheight{1.25}\smash{\begin{tabular}[t]{l}$B$\end{tabular}}}}%
    \put(0,0){\includegraphics[width=\unitlength,page=3]{braid.pdf}}%
  \end{picture}%
\endgroup%

                \caption{}
                \Description{The parts of the link remaining on the pages are pushed off to become transverse to the pages.}
                \label{fig:small perturbation}
            \end{subfigure}
        \caption{From a generalized arc presentation of $L$ in $\operatorname{OB}$ to a braiding of $L$ relative $\operatorname{OB}$.}
    \end{figure}

    Now let $B_1$ and $B_2$ be the bindings of the open books $\operatorname{OB}_1$ and $\operatorname{OB}_2$, respectively. Since $B_1$ and $B_2$ are in a mutual arc presentation, we are given a generalized arc presentation of $B_1$ relative to $\operatorname{OB}_2$. After isotoping $B_1$ as above to be braided relative to $\operatorname{OB}_2$, it remains to show that $B_2$ can be isotoped, while fixing $B_1$, such that $B_2$ becomes braided relative to $\operatorname{OB}_1$.

The initial situation, before the isotopy of $B_1$, is depicted in Figures~\ref{fig:nbhd of p_i1} and~\ref{fig:nbhd of p_i2}. Near each intersection point of $B_1$ and $B_2$, choose a sufficiently small sphere $S$. The behavior of an open book near its binding implies that the pages of $\operatorname{OB}_1$ intersect $S$ in lines of longitude when the intersection points of $S$ with $B_1$ are regarded as the north and south poles. Likewise, the pages of $\operatorname{OB}_2$ intersect $S$ in lines of longitude when the intersection points of $S$ with $B_2$ are regarded as the north and south poles, see Figure~\ref{fig:nbhd of p_i2}.

The isotopy of $B_1$ described above (Figures~\ref{fig:arc presentation} and~\ref{fig:arc on the page}) can be taken to occur in the ball bounded by $S$. After the isotopy, the local segment of $B_1$ is moved to an equatorial circle on $S$. In particular, it becomes positively transverse to the lines of longitude on $S$, which are the intersections of the pages of $\operatorname{OB}_2$ with $S$. This isotopy extends to an isotopy of the open book $\operatorname{OB}_1$. Moreover, we may assume that, outside a neighborhood of the deformed arc on $S$, the intersections of the pages of $\operatorname{OB}_1$ with $S$ remain unchanged, see Figure~\ref{fig:isotopy1}.

The intersection points of $B_2$ with $S$ divide the boundary of a transverse disk to $B_1$ into two arcs. The arc whose orientation agrees with that of $B_2$ is disjoint from the deformed arc of $B_1$. Isotoping the corresponding local segment of $B_2$ along this disk to its boundary therefore yields a curve that is positively transverse to the pages of $\operatorname{OB}_1$, as shown in the Figure~\ref{fig:isotopy2}. Similarly, this isotopy of $B_2$ preserves $B_1$ being positively transverse to the pages of $\operatorname{OB}_2$. Finally, a small perturbation of $B_2$, analogous to that in Figure~\ref{fig:small perturbation}, makes $B_2$ everywhere positively transverse to the pages of $\operatorname{OB}_1$.
    \begin{figure}[ht]
    \newcommand{\Description}[1]{}
        \begin{subfigure}{0.24\textwidth}
                \renewcommand\captionlabelfont{}
                \centering
                \includegraphics[width=\textwidth]{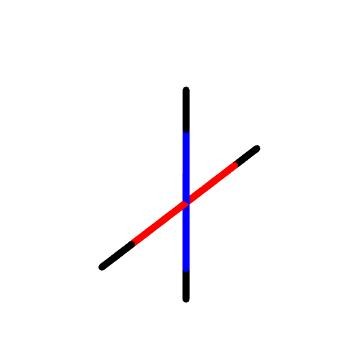}
                \caption{}
                \Description{Two intersecting line segments, one red and one blue, represent part of the bindings of two open books.}
                \label{fig:nbhd of p_i1}
        \end{subfigure}
        \begin{subfigure}{0.24\textwidth}
                \renewcommand\captionlabelfont{}
                \centering
                \includegraphics[width=\textwidth]{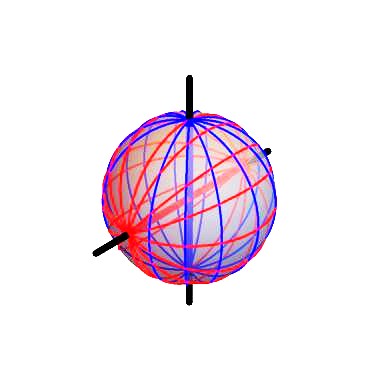}
                \caption{}
                \Description{Longitudes are formed when intersecting the pages of the open book with a sphere centered at the intersection of two bindings. A longitude induced by one open book is transverse to a longitude induced by the other open book.}
                \label{fig:nbhd of p_i2}
        \end{subfigure}
        \begin{subfigure}{0.24\textwidth}
                \renewcommand\captionlabelfont{}
                \centering
                \includegraphics[width=\textwidth]{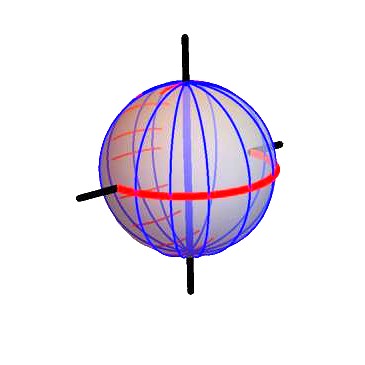}
                \caption{}
                \Description{The red line segment is deformed to a longitude on the sphere, transverse to the blue longitudes.}
                \label{fig:isotopy1}
        \end{subfigure}
        \begin{subfigure}{0.24\textwidth}
                \renewcommand\captionlabelfont{}
                \centering
                \includegraphics[width=\textwidth]{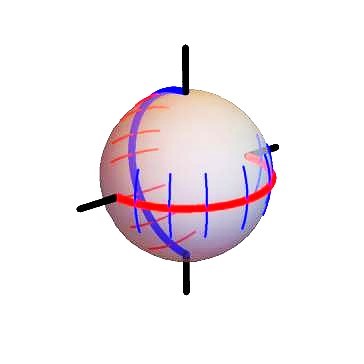}
                \caption{}
                \Description{The blue line segment is deformed to a longitude on the sphere, transverse to the red longitudes.}
                \label{fig:isotopy2}
        \end{subfigure}
        \caption{The deformation of $B_1$ and $B_2$ near an intersection point. a) $B_1$ and $B_2$ intersect in an isolated point. One segment is colored red, the other blue. b) A sphere around the intersection point. Red curves indicate the intersections between fibers of one open book with the sphere. Blue curves are the intersections between the fibers of the other open book and the sphere. c) Deforming the red segment of the binding to a curve that is transverse to the blue fibers. d) Deforming the blue segment of the other binding to be transverse to the red fibers.}
        \label{fig:nbhd of p_i}
    \end{figure}   
\end{proof}

\section{Mutual arc presentations of canonically fibered links}\label{sec:canonical}
There exist fibered links whose fiber surface, or equivalently, minimal genus Seifert surface, cannot be obtained using Seifert's algorithm for any diagram of the link. We say that a fibered link $B\subset S^3$ is \emph{canonically fibered} if the fiber surface of $B$ can be obtained via Seifert's algorithm for some link diagram of $B$. For such links, we show that they are not counterexamples to Conjecture~\ref{conj:mutual_arc} by providing an algorithm that constructs mutual arc presentations for them.

\begin{lemma}\label{lemma:mutual page presentation}
    Given any canonical Seifert surface $\mathcal{S}$ of a non-split link $L\subset S^3$, there exists a mutual arc presentation of $L$ and the unknot relative to the disk and $\mathcal{S}$. 
\end{lemma}
\begin{proof}
Let $\mathcal{D}_L$ be a diagram of $L$ such that the given canonical Seifert surface $\mathcal{S}$ is obtained via Seifert's algorithm using $\mathcal{D}_L$.

Given a $4$-valent, planar graph, by \emph{smoothing} a vertex we mean the process of removing a vertex and reconnecting adjacent edges in one of the two ways shown in Figure~\ref{fig:smoothings}.
\begin{figure}[ht]
    \newcommand{\Description}[1]{}
    \centering
    \includegraphics[width=0.4\linewidth]{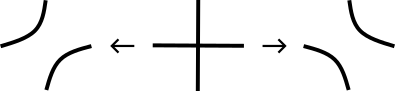}
    \caption{Two ways of smoothing a vertex of a $4$-valent planar graph.}
    \Description{Near a vertex of a 4-valent graph, the four incident branches form a cross. The vertex is replaced by two disjoint arcs in one of two ways: either connecting the north branch to the west branch and the east branch to the south branch, or connecting the north branch to the east branch and the west branch to the south branch.}
    \label{fig:smoothings}
\end{figure}

We construct a planar unknot $U$, to be the binding of an unbook of $S^3$, by choosing a smoothing\footnote{One way to see that for any connected, planar, $4$-valent graph there exists a choice of smoothings at each vertex that converts the graph into a planar unknot is to use a strengthened form of Euler’s theorem on bridge problems. The strengthened form guarantees the existence of an Eulerian cycle -- a closed path that visits every edge exactly once -- in which, at every vertex, the path turns either left or right (but never goes straight)~\cite{eulerian_path}. In this setting, turning left or right at a vertex corresponds precisely to choosing a smoothing.} of the following connected, planar $4$-valent graph. The process is pictured in Figure~\ref{fig:algorithm}.

\begin{enumerate}[(i)]
    \item
    Draw $4$ edges that form a closed loop around each crossing of $\mathcal{D}_L$. Such an edge is called \emph{long} (resp.\ \emph{short}) if its two endpoints lie on two distinct (resp.\ the same) Seifert circles of $\mathcal{S}$ obtained by applying Seifert's algorithm to $\mathcal{D}_L$.
    \item
    Add $2$ edges that form a closed loop enclosing each part of an arc of $\mathcal{D}_L$ outside the closed loops drawn in the previous step, such that, together with the edges drawn in the previous step, a connected, planar $4$-valent graph is formed. We say these edges are \emph{parallel} ones.
\end{enumerate}

By an arc of $L$, we mean a (connected) segment between two intersection points of $U$ and $L$. Observe that there are no crossings between interiors of arcs of $L$ and the unknot. Furthermore, two arcs of $L$ share at most one crossing, therefore we can assume each arc of $L$ lies on a different page of the unbook with binding $U$, i.e.\ it is an arc presentation of $L$ with binding of the unbook given by $U$.

Claim: the edges of the unknot $U$ can be isotoped to lie on pages of a half open book with page $\mathcal{S}$, i.e.\ such an arc presentation of $L$ relative to the unbook with binding $U$ can be isotoped to a mutual arc presentation of $L$ and $U$.

First isotope short edges and parallel edges to each lie on a separate page of a half open book with page $\mathcal{S}$. If a short edge is already lying on $\mathcal{S}$ it is necessarily disjoint from parallel edges that are also lying on $\mathcal{S}$ by our construction. If a short edge is not already lying on $\mathcal{S}$, then rotate it around $L$ to lie on a page of the half open book with page $\mathcal{S}$ as in the last part of Figure~\ref{fig:algorithm}. At every vertex of the 4-valent graph discussed earlier there are one short edge, one long edge and two parallel edges that meet at the vertex. Therefore, the short edges that are not already on $\mathcal{S}$ can be rotated around $L$ to lie on a page of the half open book with page $\mathcal{S}$ without passing one edge through another. Similarly, after this, we can rotate the parallel edges that do not already lie on $\mathcal{S}$ to lie on a page of the half open book with page $\mathcal{S}$ without creating any intersections. 

\begin{figure}[ht]
    \newcommand{\Description}[1]{}
    \centering
    \def\svgwidth{0.95\textwidth}
\begingroup%
  \makeatletter%
  \providecommand\color[2][]{%
    \errmessage{(Inkscape) Color is used for the text in Inkscape, but the package 'color.sty' is not loaded}%
    \renewcommand\color[2][]{}%
  }%
  \providecommand\transparent[1]{%
    \errmessage{(Inkscape) Transparency is used (non-zero) for the text in Inkscape, but the package 'transparent.sty' is not loaded}%
    \renewcommand\transparent[1]{}%
  }%
  \providecommand\rotatebox[2]{#2}%
  \newcommand*\fsize{\dimexpr\f@size pt\relax}%
  \newcommand*\lineheight[1]{\fontsize{\fsize}{#1\fsize}\selectfont}%
  \ifx\svgwidth\undefined%
    \setlength{\unitlength}{1254.67517763bp}%
    \ifx\svgscale\undefined%
      \relax%
    \else%
      \setlength{\unitlength}{\unitlength * \real{\svgscale}}%
    \fi%
  \else%
    \setlength{\unitlength}{\svgwidth}%
  \fi%
  \global\let\svgwidth\undefined%
  \global\let\svgscale\undefined%
  \makeatother%
  \begin{picture}(1,0.51200632)%
    \lineheight{1}%
    \setlength\tabcolsep{0pt}%
    \put(0,0){\includegraphics[width=\unitlength,page=1]{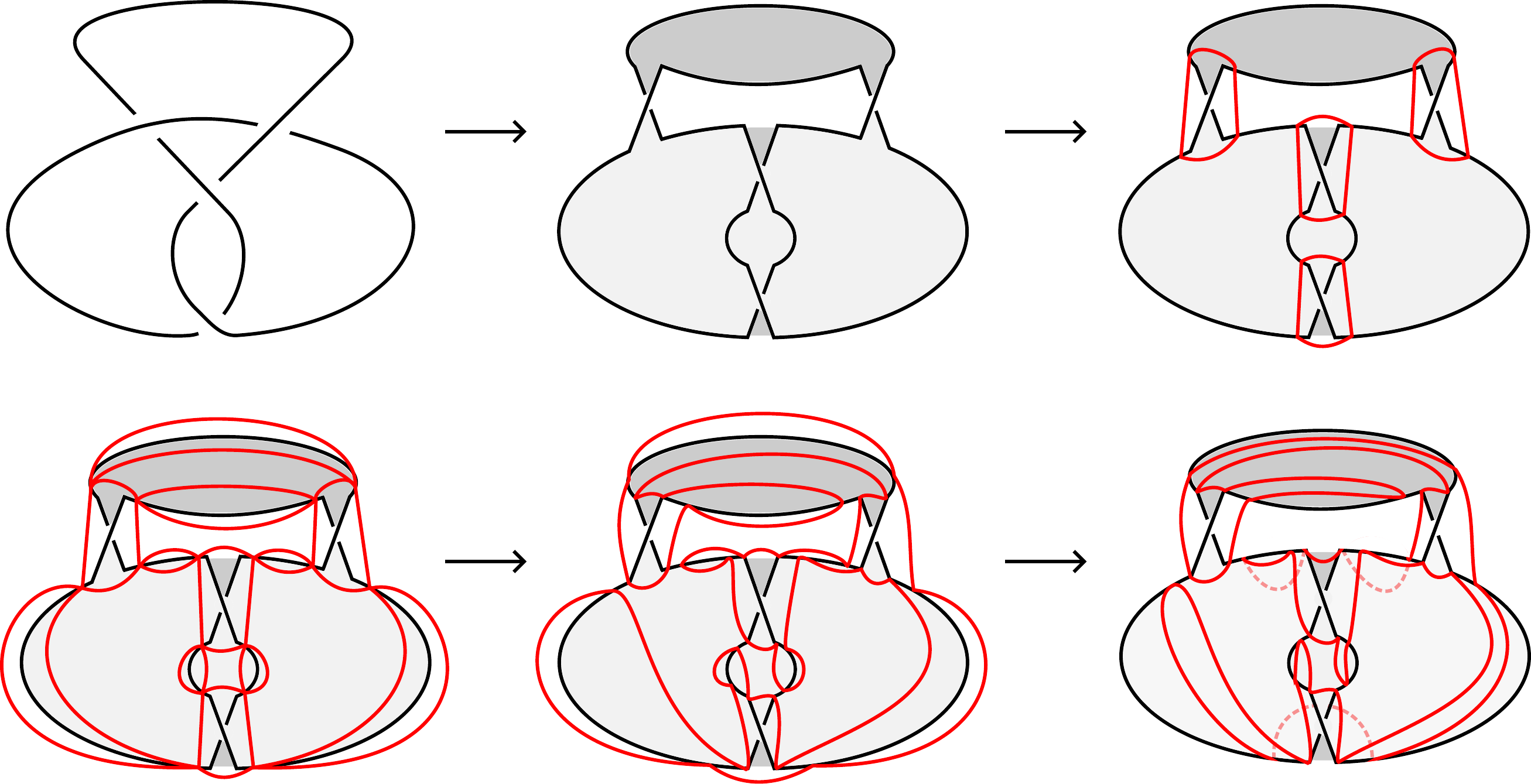}}%
    \put(0.26444091,0.48176041){\makebox(0,0)[lt]{\lineheight{1.25}\smash{\begin{tabular}[t]{l}$\text{Seifert's}$\end{tabular}}}}%
    \put(0.26444091,0.45067665){\makebox(0,0)[lt]{\lineheight{1.25}\smash{\begin{tabular}[t]{l}$\text{ algorithm}$\end{tabular}}}}%
  \end{picture}%
\endgroup%

    \caption{Constructing the binding of an unbook and deforming the short and parallel edges.}
    \Description{Starting with a minimal crossing diagram of the figure-eight knot, a Seifert surface is obtained by Seifert's algorithm. Four red rectangles enclose the four bands connecting Seifert circles. More red arcs are added to form a four-valent graph, which is then smoothed out to a planar unknot. Finally, arcs of the unknot are isotoped to lie on the Seifert surface of the figure-eight knot.}
    \label{fig:algorithm}
\end{figure}

Finally, as depicted in Figure~\ref{fig:long edge}, we deform each long edge towards the interior of the half-twisted band in $\mathcal{S}$. The figure depicts this deformation for a band with a positive half-twist. For a negative half-twist the corresponding arguments apply. Note that all arcs in the second row of figures in Figure~\ref{fig:long edge} lie on a page of a half open book with page $\mathcal{S}$. This can be seen as follows.

The two sides of the (orientable) Seifert surface are colored in different shades of grey, a darker one and a lighter one. In both columns of Figure~\ref{fig:long edge} the left arcs in the second row are the result of pushing a curve on $\mathcal{S}$ off the surface in the direction of the darker side. Likewise, the right arcs are the results of pushing a curve on $\mathcal{S}$ off the surface in the direction of the light side, see also Remark~\ref{rem:ref} below. 

In particular, Figure~\ref{fig:long edge} shows that all long edges can be deformed to lie on pages of a half open book with page $\mathcal{S}$, so that we obtain a generalized arc presentation of the unknot relative to $\mathcal{S}$ and therefore a mutual arc presentation of $L$ and $U$.
\end{proof}

\begin{figure}[ht]
    \newcommand{\Description}[1]{}
    \centering
    \includegraphics[width=0.5\linewidth]{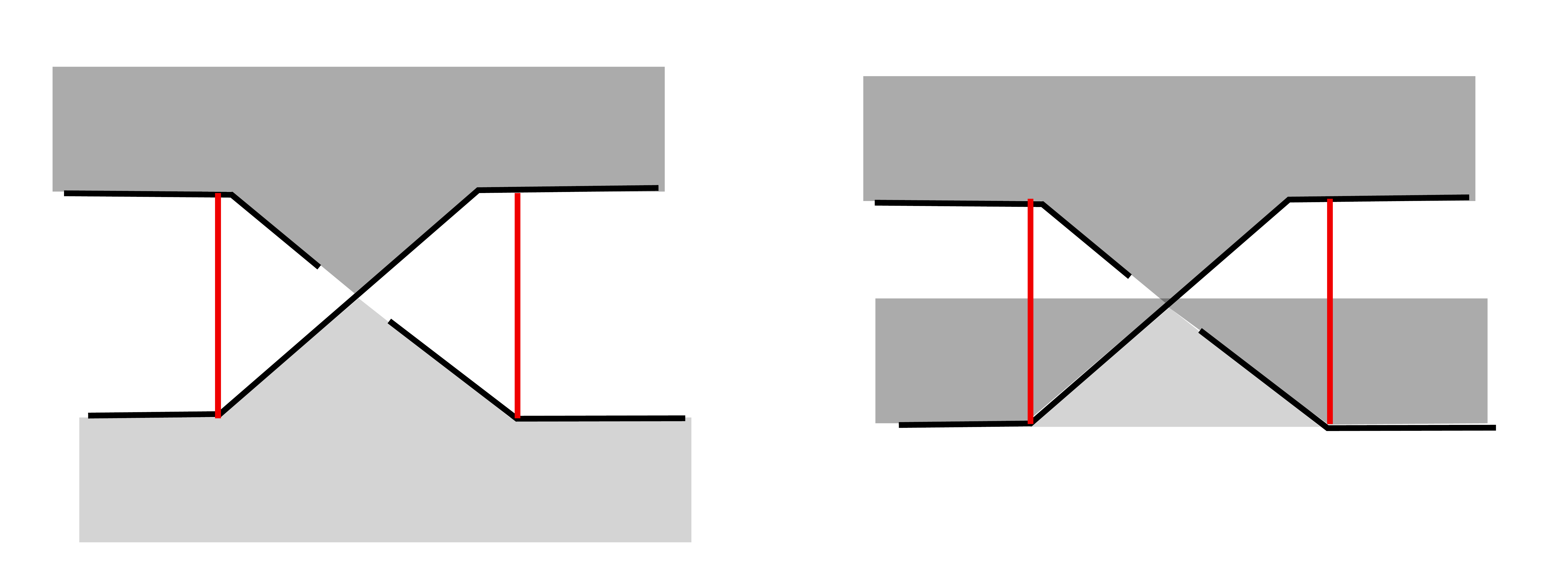}
    \includegraphics[width=0.5\linewidth]{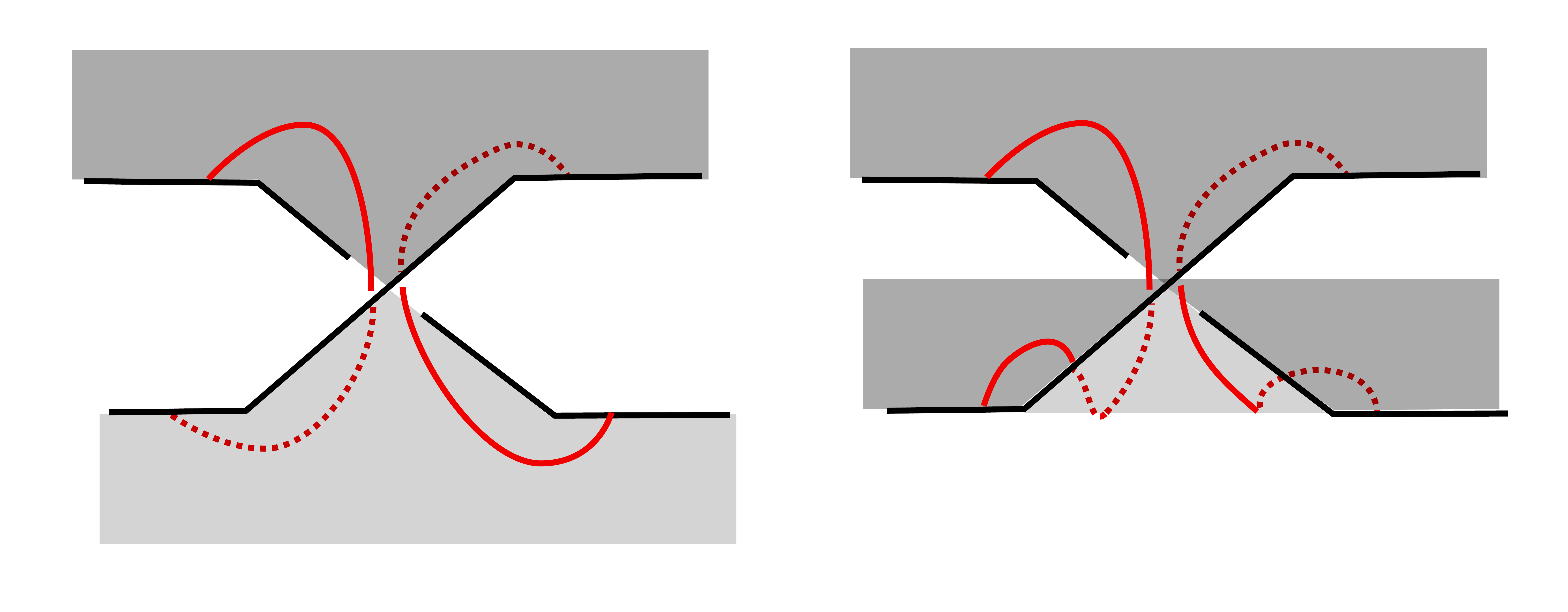}
    \caption{Deformation of long edges. The first row of figures shows the initial arcs and the canonical surface $\mathcal{S}$. In the second row the arcs are deformed to lie on two pages of a half open book of the canonical surface $\mathcal{S}$ (both of which are not drawn).}
    \Description{Deformations of long arcs are shown in different possible canonical Seifert surfaces.}
    \label{fig:long edge}
\end{figure}

\begin{remark}\label{rem:ref}
In the last step of the proof, where we explain that the deformed arcs lie on pages of a half open book with page $\mathcal{S}$, it is important which side of the surface these arcs lie on, that is to say, which arc lies above the darker side and which arc lies above the lighter side. 
\begin{figure}
    \newcommand{\Description}[1]{}
    \centering
    \includegraphics[width=0.65\linewidth]{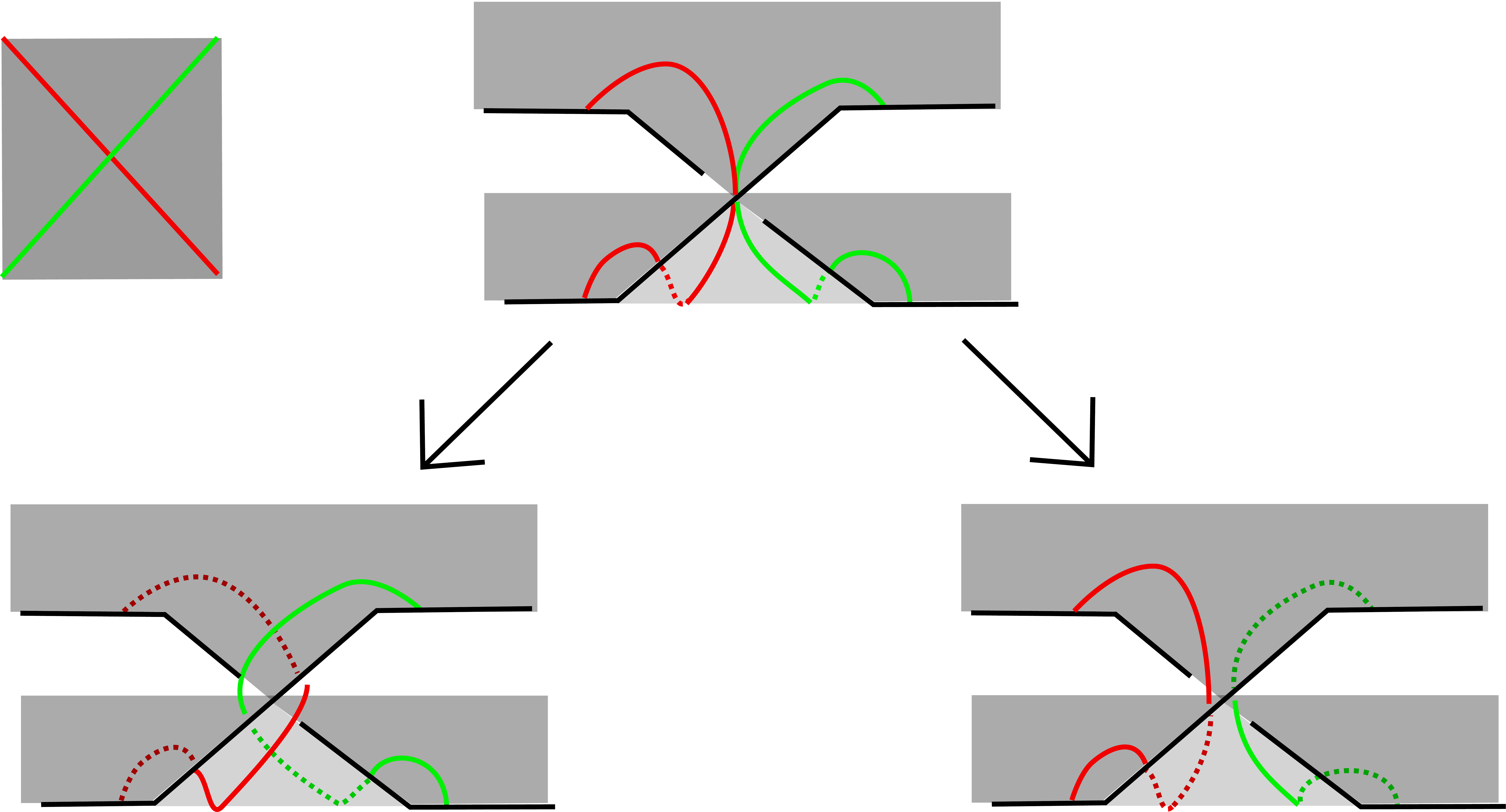}
    \caption{The arcs from Figure~\ref{fig:long edge} are obtained by pushing two arcs on $\mathcal{S}$ off the surface in appropriate directions. Pushing them in the opposite directions, as in the bottom left, results in a pair of arcs with a full twist between them.}
    \Description{Two isotopies are shown side by side for comparison. In both, the arcs are pushed off the surface. In the first isotopy, the arcs become tangled, whereas in the second, they remain untangled.}
    \label{fig:ref}
\end{figure}
Figure~\ref{fig:ref} shows two curves on a rectangle. In order to distinguish the two curves more easily, they are colored differently, one red and one green. Next, the rectangle is twisted, so that it appears as a half-twisted band and part of the Seifert surface $\mathcal{S}$. We thus have two curves (which intersect each other) on $\mathcal{S}$. We now push these arcs in opposite directions off the surface, as in the bottom part of the figure. Pushing the red curve in the direction of darker side and the green curve in the direction of the lighter side produces the figure on the bottom right of Figure~\ref{fig:ref}, which is the same as the bottom right figure in Figure~\ref{fig:long edge}. In particular, the red and green curve in the bottom right figure of Figure~\ref{fig:ref} lie on pages of a half open book with page $\mathcal{S}$ and are isotopic to the two parallel arcs in the upper right of Figure~\ref{fig:long edge}. However, if we had pushed off the arcs in the opposite directions, the red curve in the direction of the lighter and the curve in the direction of the darker side, we would have obtained the arcs in the bottom left figure of Figure~\ref{fig:ref}. They also lie on pages of a half open book with page $\mathcal{S}$, but have a full twist between and are therefore not isotopic to the two parallel arcs in the upper right of Figure~\ref{fig:long edge}. Analogous considerations apply to the left column of Figure~\ref{fig:long edge} and half-twisted bands with the opposite handedness.
\end{remark}

\begin{remark}
    By construction, the arc presentation of $L$ obtained in the proof of Lemma~\ref{lemma:mutual page presentation} consists of $8c$ arcs, where $c$ is the number of crossings of $\mathcal{D}_L$. This is not expected to be the minimal number of arcs among all arc presentations of a given link. For example, the minimal such number, called \emph{arc index}, for an alternating knot is equal to its minimal crossing number plus two~\cite{arc_index_of_alternating_knots, bounding_arc_index}.
\end{remark}

\begin{remark}
In our definition of a generalized arc presentation of a link (Definition~\ref{def:mutual arc}) we do not require different arcs to lie on different pages, which obviously does not affect the concept. We can thus define the \emph{page index} of a link $L$ relative to a Seifert surface $\mathcal{S}_2$ of some other link $L_2$ as the minimal number of pages of a half open book with page $\mathcal{S}_2$ that contain arcs of a generalized arc presentation of $L$ relative to $\mathcal{S}_2$. Likewise, the \emph{mutual page index} of a pair of Seifert surfaces $\mathcal{S}_1$ and $\mathcal{S}_2$ with boundaries $L_1$ and $L_2$, respectively, is the minimum of the minimal number of pages of a half open book with page $\mathcal{S}_2$ that contain arcs of $L_1$ and the minimal number of pages of a half open book with page $\mathcal{S}_1$ that contain arcs of $L_2$ in a mutual arc presentation of $L_1$ and $L_2$ relative to $\mathcal{S}_2$ and $\mathcal{S}_1$

The construction in the proof of Lemma~\ref{lemma:mutual page presentation} shows that the page index of any link $L$ relative to a disk is at most 5. Likewise the mutual page index of any canonical Seifert surface and the disk is at most 5. This can be seen as follows. There are short edges and parallel edges of the unknot that lie already on $\mathcal{S}=\mathcal{S}_0$ from the start. When we rotate the remaining short edges and parallel edges around $L$ until they lie on a page of the half open book with page $\mathcal{S}$, we choose the rotation so that the deformed edge lies below the Seifert disk whose boundary contains its endpoints. There are no intersections between different arcs, so that each of these deformed arcs can be taken to lie on $\mathcal{S}_{\varepsilon}$ or $\mathcal{S}_{-\varepsilon}$ for some $\varepsilon$. Finally, all deformed long edges can be taken to lie on $\mathcal{S}_{2\varepsilon}$ or $\mathcal{S}_{-2\varepsilon}$.
\end{remark}

\begin{proof}[Proof of Theorem~\ref{thm:gc}]
Suppose the binding $B$ of an open book $\operatorname{OB}$ is canonically fibered. Then there exists a diagram of $B$ such that the result of applying Seifert's algorithm to the diagram is a page $\mathcal{S}_0$ of $\operatorname{OB}$. By Lemma~\ref{lemma:mutual page presentation}, there exists a mutual arc presentation of $B$ and the unknot relative to the disk and $\mathcal{S}_0$. Since $\mathcal{S}_0$ is a page of $\operatorname{OB}$ all pages of its half open book are also pages of $\operatorname{OB}$, and we have a mutual arc presentation of $\operatorname{OB}$ and the unbook.
\end{proof}

Proposition~\ref{prop:mutual book presentation to mutually braided} implies that there are isotopies of $B$ and of the unknot that turn them into mutually braided links from a mutual arc presentation.
\begin{corollary}\label{cor:gc}
    Let $B$ be a canonically fibered link in $S^3$, then $B$ and the unknot can be mutually braided. \qed
\end{corollary}

\begin{proof}[Proof of Corollary~\ref{cor:alternating}]
    Since a minimal genus Seifert surface is always obtained via Seifert's algorithm when applied to an alternating diagram~\cite{alternating1, alternating2}, alternating fibered links are all canonically fibered.

    The second part of the corollary follows directly from~\cite{Stoimenow}. Stoimenow provides infinitely many almost alternating fibered knots that are canonically fibered, yet the minimal genus Seifert surface is not obtained by applying Seifert's algorithm to an almost alternating diagram.
\end{proof}

\begin{proof}[Proof of Corollary~\ref{cor:deplumbing}]
    In his Ph.D. thesis~\cite{Misev_2016}, Misev demonstrates the existence of an infinite family of Alexander polynomials which, when realized by a two component, genus one link, cannot arise from Hopf plumbings. The explicit realization of these polynomials by hyperbolic, fibered links of two components and genus one is due to Stoimenow~\cite{Stoimenow_infinite_family}. The upper bound on the canonical genus for each link in this family, as directly obtained from the diagrams in Figure~3.1 in~\cite{Misev_2016}, is one; therefore, the infinite family of links are all canonically fibered.
\end{proof}


By defining mutual arc presentations in terms of half open books of Seifert surfaces as in Defintion~\ref{def: generalized mutual arc}, we have obtained results on links that are not necessarily fibered. In Definition~\ref{def:braided}, we have adopted the more restrictive viewpoint, defining mutual braiding of two bindings of open books. It is also possible to extend this definition to non-fibered links. Let $L_1$, $L_2$ be two disjoint, oriented links in $M$ and let $f_1\colon M\backslash L_1\to S^1$, $f_2\colon M\backslash L_2\to S^1$ be circle-valued Morse maps that have the same behavior as an open book on a tubular neighborhood of $L_1$ and $L_2$, respectively. Let $\mathcal{S}_1$ and $\mathcal{S}_2$ be Seifert surfaces of $L_1$ and $L_2$, respectively, that are regular level sets of $f_1$ and $f_2$, respectively. Then nearby level sets of $f_i$ are isotopic to $\mathcal{S}_i$, $i=1,2$, and in particular the nearby level sets are exactly the pages of a half open book with page $\mathcal{S}_i$. We may then define $L_1$ and $L_2$ to be \emph{mutually braided with respect to $f_1$ and $f_2$} if $L_1$ is positively transverse to the level sets of $f_2$ and $L_2$ is positively transverse to the level sets of $f_1$. In the special case where $L_1$ and $L_2$ are fibered links and $f_1$ and $f_2$ are the corresponding fibration maps, this reduces to our previous definition of mutual braiding between two fibered links, or, equivalently, a mutual braiding of the two open books.

The proof of Proposition~\ref{prop: detailed mutual book presentation to mutually braided} is unaffected by this extension of the definition, since the deformation of each of the two links occurs exclusively in a tubular neighborhood of the other. We thus obtain the following result.
\begin{proposition}
Let $L_1$, $L_2$, $\mathcal{S}_1$, $\mathcal{S}_2$, $f_1$ and $f_2$ be as above. If there is a mutual arc presentation of $L_1$ and $L_2$ relative to $\mathcal{S}_2$ and $\mathcal{S}_1$, then $L_1$ and $L_2$ can be deformed to be mutually braided with respect to $f_1$ and $f_2$. \qed
\end{proposition}

Likewise, Corollary~\ref{cor:gc} becomes:
\begin{proposition}\label{prop:morse}
Let $L$ be an oriented link in $S^3$, and let $f\colon S^3\backslash L\to S^1$ be a circle-valued Morse map that has the same behavior as an open book in a tubular neighborhood of $L$. Suppose that there exists a regular level set of $f$ that is a canonical Seifert surface of $L$. Then $L$ and the unknot can be mutually braided with respect to $f$ and the fibration map of the unbook. \qed
\end{proposition}

Since every canonical Seifert surface is free, it follows from Theorem 11 in \cite{hirasawarudolph} that every canonical Seifert surface permits a Morse map as in Proposition~\ref{prop:morse}. In particular, for every given link there exist infinitely many distinct surfaces as in the proposition. In fact, for every given braid word in Artin generators, the proofs of Proposition~2.2 and Corollary~2.3 in \cite{bodehirasawa} illustrate how to construct an analytic expression for $f$ as in Proposition~\ref{prop:morse} such that a regular level set of $f$ is the braided surface corresponding to the given braid word (where each strand of the braid corresponds to a disk and each letter in the word corresponds to a half-twisted band). Note that all Seifert surfaces that can be represented as braided surfaces corresponding to words in Artin generators (rather than the more general band generators) are canonical, thus producing infinitely many examples of Proposition~\ref{prop:morse}.

\let\MRhref\undefined
\bibliographystyle{hamsalpha}  
\bibliography{Sources}

\end{document}